\documentclass{amsart}
\usepackage{amssymb}
\usepackage{amsfonts,color}
\usepackage{graphicx}
\usepackage{hyperref}
\newtheorem{theorem}{Theorem}[section]
\theoremstyle{plain}

\newtheorem{corollary}{Corollary}[section]

\newtheorem{definition}{Definition}[section]
\newtheorem{example}{Example}[section]

\newtheorem{lemma}{Lemma}[section]

\newtheorem{proposition}{Proposition}[section]
\newtheorem{remark}{Remark}[section]

\numberwithin{equation}{section}
\input{tcilatex}

\begin{document}
\title[]{on symplectic self-adjointness of Hamiltonian operator matrices}
\author{Alatancang Chen, Guohai Jin,  and Deyu Wu}
\address{School of Mathematical Sciences\\ Inner Mongolia University\\ Hohhot\\ 010021\\ China}
\email{alatanca@imu.edu.cn, ghjin2006@gmail.com, wudeyu2585@163.com}
\address{The corresponding author: Guohai Jin}
\thanks{This paper is in final form and no version of it will be submitted
for publication elsewhere.}
\thanks{Fax:+86-471-4991650}

\begin{abstract}
Symplectic self-adjointness of Hamiltonian operator matrices is studied,
which is important to symplectic elasticity and optimal control.
For the cases of diagonal domain and off-diagonal domain,
necessary and sufficient conditions are shown.
The proofs use Frobenius-Schur fractorizations of unbounded operator matrices.
Under additional assumptions, sufficient conditions based on perturbation method are obtained.
The theory is applied to a problem in symplectic elasticity.
\end{abstract}
\maketitle

\section{introduction}
There are a number of very interesting ways that \emph{Hamiltonian operator matrices} (see Definition \ref{def2.2} below) can arise.
We mention a few.
First, many linear boundary value problems in mathematical physics can be written as the Hamiltonian system
(or Hamiltonian equation) $\dot{u}=Hu+f$, where
\begin{align}\label{eq1.1}
H=\left(
    \begin{array}{cc}
      A & B \\
      C & -A^* \\
    \end{array}
  \right)
\end{align}
is a Hamiltonian operator matrix acting on the product space $X\times X$ of some Hilbert space $X$,
so that the solvability of the original boundary value problem is reduced to spectral properties of the Hamiltonian operator matrix $H$,
see e.g. \cite{Atk, Krall2002} for ordinary differential equations and \cite{CM, Olver, ZAZ} for partial differential equations.
This is a typical case in \emph{symplectic elasticity}.
In elasticity, symplectic approach (i.e., Hamiltonian system approach) was first applied in the
early 1990s by Professor Wanxie Zhong,
see \cite{LX,YZL,Zhong} and the references therein.
The new approach is efficient for solving basic problems in solid mechanics which have long
been bottlenecks in the history of elasticity and, moreover, analytical solutions could
be obtained by expansion of eigenfunctions.
Second, Hamiltonian operator matrices also arise in theory of optimal control.
It is well known that the solutions $U$ of the Riccati equation
$$A^*U+UA+UBU-C=0$$
are in one-to-one correspondence with graph subspaces that are invariant under the operator matrix $H$ given by \eqref{eq1.1},
where $A, B, C$ are unbounded linear operators and $B, C$ are nonnegative,
see e.g. \cite{TW, Wyss2011} and the references therein.
Therefore, spectral properties of Hamiltonian operator matrices have drawn a lot of interest,
see \cite{AHF,ADG,AKK,KZ,Ku1,KM,LRW,QC1,QC2,SS,WA}.

The property \emph{symplectic self-adjointness} (see Definition \ref{def2.3} below)
is distinctive for certain Hamiltonian operator matrices.
It is only for symplectic self-adjoint Hamiltonian operator matrices that the spectral theorems hold \cite{AHF}
and it is only symplectic self-adjoint Hamiltonian operator matrices that may be invertible
which is sometimes important in the investigation of Hamiltonian equations \cite{Ku1}.
This paper is devoted to studying methods for proving that Hamiltonian operator matrices are symplectic self-adjoint.
For the case of $\mathcal{D}(H)=\mathcal{D}(A)\times\mathcal{D}(A^*)$ or $\mathcal{D}(H)=\mathcal{D}(B)\times\mathcal{D}(D)$,
Wu and Alatancang \cite{WA} proved the following basic perturbation result.
\begin{theorem}\label{th1.1}
Let
$$H=\left(
      \begin{array}{cc}
        A & B \\
        C & -A^* \\
      \end{array}
    \right)$$
be a Hamiltonian operator matrix. Then $H$ is symplectic self-adjoint if one of the following holds:
\begin{enumerate}
\item $C$ is $A$-bounded and $B$ is $A^*$-bounded with both relative bounds less than $1$,
\item $A$ is $C$-bounded and $A^*$ is $B$-bounded with both relative bounds less than $1$.
\end{enumerate}

\end{theorem}
In Section \ref{sec3}, by the \emph{Frobenius-Schur fractorization} technique,
we shall establish some necessary and sufficient conditions for a Hamiltonian operator matrix to be symplectic self-adjoint
which extend Theorem \ref{th1.1} to non-perturbation cases.
Moreover, we also obtain some perturbation results under the assumptions that are different from or weaker than those in Theorem \ref{th1.1}.
Finally, we shall apply the new results to a problem in symplectic elasticity which seems cannot be resolved by the previously published methods,
see Section \ref{sec4} below.

\section{preliminaries}\label{sec2}

Our notion of an operator matrix is taken from \cite[p. 97]{Wyss2008}, see also \cite[Section 2.2]{Tre} for another definition.
\begin{definition}\label{def2.1}
Let $X_1, X_2$ be Banach spaces and consider linear operators $A:\mathcal{D}(A)\subset X_1\to X_1,
B:\mathcal{D}(B)\subset X_2\to X_1, C:\mathcal{D}(C)\subset X_1\to X_2$, and $D:\mathcal{D}(D)\subset X_2\to X_2$.
Then the matrix
$$H=\left(
      \begin{array}{cc}
        A & B \\
        C & D \\
      \end{array}
    \right)$$
is called a block operator matrix on $X_1\times X_2$.
It induces a linear operator on $X_1\times X_2$ which is also denoted by $H$:
\begin{align*}
\mathcal{D}(H):&=(\mathcal{D}(A)\cap\mathcal{D}(C))\times(\mathcal{D}(B)\cap\mathcal{D}(D)),\\
H\left(
   \begin{array}{c}
     x_1 \\
     x_2 \\
   \end{array}
 \right):&=\left(
             \begin{array}{c}
               Ax_1+Bx_2 \\
               Cx_1+Dx_2 \\
             \end{array}
           \right) \mbox{~for~} \left(
                                  \begin{array}{c}
                                    x_1 \\
                                    x_2 \\
                                  \end{array}
                                \right)\in\mathcal{D}(H).
\end{align*}
\end{definition}

The following \emph{Frobenius-Schur fractorization} will play an important role in the proofs of our main theorems.
\begin{lemma}(\cite[Section 2.2]{Tre})\label{lem2.1}
Let
\begin{align*}
H=\left(
    \begin{array}{cc}
      A & B \\
      C & D \\
    \end{array}
  \right)
\end{align*}
be a block operator matrix acting on the product space $X\times X$ of some Banach space $X$.
\begin{enumerate}
\item Suppose that $D$ is closed with $\rho(D)\neq\emptyset$, and that $\mathcal{D}(D)\subset\mathcal{D}(B)$.
Then for some (and hence for all) $\lambda\in\rho(D)$,
\begin{align*}
H-\lambda=&\left(
            \begin{array}{cc}
              I & B(D-\lambda)^{-1} \\
              0 & I \\
            \end{array}
          \right)\left(
                   \begin{array}{cc}
                     S_1(\lambda) & 0 \\
                     0 & D-\lambda \\
                   \end{array}
                 \right)\\
                 &\left(
                          \begin{array}{cc}
                            I & 0 \\
                            (D-\lambda)^{-1}C & I \\
                          \end{array}
                        \right),
\end{align*}
where $S_1(\lambda):=A-\lambda-B(D-\lambda)^{-1}C$ is the first Schur complement of $H$ with domain $\mathcal{D}(S_1(\lambda))=\mathcal{D}(A)\cap\mathcal{D}(C)$.
\item Suppose that $A$ is closed with $\rho(A)\neq\emptyset$, and that $\mathcal{D}(A)\subset\mathcal{D}(C)$.
Then for some (and hence for all) $\lambda\in\rho(A)$,
\begin{align*}
H-\lambda=&\left(
            \begin{array}{cc}
              I & 0 \\
              C(A-\lambda)^{-1} & I \\
            \end{array}
          \right)\left(
                   \begin{array}{cc}
                     A-\lambda & 0 \\
                     0 & S_2(\lambda) \\
                   \end{array}
                 \right)\\
                 &\left(
                          \begin{array}{cc}
                            I & (A-\lambda)^{-1}B \\
                            0 & I \\
                          \end{array}
                        \right),
\end{align*}
where $S_2(\lambda):=D-\lambda-C(A-\lambda)^{-1}B$ is the second Schur complement of $H$ with domain $\mathcal{D}(S_2(\lambda))=\mathcal{D}(B)\cap\mathcal{D}(D)$.
\end{enumerate}
\end{lemma}

We use the following definition of a Hamiltonian operator matrix, see also \cite{AHF, ADG, Wyss2011} for other definitions.
\begin{definition}\label{def2.2}
Let $X$ be a complex Hilbert space.
A Hamiltonian operator matrix is a block operator matrix
\begin{align*}
H=\left(
    \begin{array}{cc}
      A & B \\
      C & -A^* \\
    \end{array}
  \right)
\end{align*}
acting on $X\times X$ with closed densely defined operators $A, B, C$
such that $B, C$ are self-adjoint and $H$ is densely defined.
\end{definition}
For a Hamiltonian operator matrix $H$, one readily checks that $JH\subset (JH)^*$, where
\begin{align*}
J=\left(
    \begin{array}{cc}
      0 & I \\
      -I & 0 \\
    \end{array}
  \right)
\end{align*}
is the unit symplectic operator matrix \cite[p. 11]{YZL}.
\begin{definition}(\cite{WA})\label{def2.3}
Let $H$ be a Hamiltonian operator matrix.
If $JH=(JH)^*$, then $H$ is called a \emph{symplectic self-adjoint Hamiltonian operator matrix}.
\end{definition}

\section{main results}\label{sec3}

In this section $H=\left(
    \begin{array}{cc}
      A & B \\
      C & -A^* \\
    \end{array}
  \right)$
will denote a Hamiltonian operator on $X\times X$
and $J=\left(
                \begin{array}{cc}
                  0 & I \\
                  -I & 0 \\
                \end{array}
              \right)$
will denote the unit symplectic operator matrix.

Firstly, we give a sufficient and necessary condition for a Hamiltonian operator matrix to be symplectic self-adjoint.
\begin{proposition}\label{prop3.1}
$H$ is symplectic self-adjoint if and only if $\mathcal{R}(H\pm iJ)=X\times X$.
\end{proposition}
\begin{proof}
Since $JH\subset(JH)^*$ and $J^*=J^{-1}=-J$,
the assertion follows from the well-known fact \cite[Theorem VIII.3]{RS1981} that
$$JH=(JH)^*\iff\mathcal{R}(JH\mp iI)=X\times X.$$
\end{proof}

Secondly, we consider the case $\mathcal{D}(H)=\mathcal{D}(A)\times\mathcal{D}(A^*)$.
\begin{theorem}\label{th3.1}
Suppose that $\mathcal{D}(H)=\mathcal{D}(A)\times\mathcal{D}(A^*)$, and that $\rho(A)\neq\emptyset$.
Then the following statements are equivalent:
\begin{enumerate}
\item $H$ is symplectic self-adjoint,
\item $A^*+\lambda+C(A-\lambda)^{-1}B=(A+\overline{\lambda}+B(A^*-\overline{\lambda})^{-1}C)^*$
for some (and hence for all) $\lambda\in\rho(A)$,
\item $A+\overline{\lambda}+B(A^*-\overline{\lambda})^{-1}C=(A^*+\lambda+C(A-\lambda)^{-1}B)^*$
for some (and hence for all) $\lambda\in\rho(A)$.
\end{enumerate}
\end{theorem}
\begin{proof}
Taking $\lambda\in\rho(A)$.
By applying Lemma \ref{lem2.1} to $(H-\lambda)$ and then a little calculation we see that
\begin{align}\label{eq3.1}
JH-\lambda J=&\left(
               \begin{array}{cc}
                 I & -C(A-\lambda)^{-1} \\
                 0 & I \\
               \end{array}
             \right)\left(
                      \begin{array}{cc}
                        -S_2(\lambda) & 0 \\
                        0 & -A+\lambda \\
                      \end{array}
                    \right)\\
\nonumber                    &\left(
                             \begin{array}{cc}
                               0 & -I \\
                               I & (A-\lambda)^{-1}B \\
                             \end{array}
                           \right),
\end{align}
where $S_2(\lambda):=-A^*-\lambda-C(A-\lambda)^{-1}B$ is the second Schur complement of $H$.
Note that $(A-\lambda)^{-1}B$ is bounded on its domain $\mathcal{D}(B)$
since $\mathcal{D}(B)\supset\mathcal{D}(A^*)$ (\cite[Proposition 3.1]{ALMS}).
Then in \eqref{eq3.1} we can replace
$(A-\lambda)^{-1}B=\overline{(A-\lambda)^{-1}B}\upharpoonright_{\mathcal{D}(B)}$ by
$\overline{(A-\lambda)^{-1}B}=(B(A^*-\overline{\lambda})^{-1})^*$ since the domain of the middle factor is equal to $\mathcal{D}(A^*)\times\mathcal{D}(A)$.
Thus
\begin{align}\label{eq3.2}
JH-\lambda J=&\left(
               \begin{array}{cc}
                 I & -C(A-\lambda)^{-1} \\
                 0 & I \\
               \end{array}
             \right)\left(
                      \begin{array}{cc}
                        -S_2(\lambda) & 0 \\
                        0 & -A+\lambda \\
                      \end{array}
                    \right)\\
\nonumber                    &\left(
                             \begin{array}{cc}
                               0 & -I \\
                               I & (B(A^*-\overline{\lambda})^{-1})^* \\
                             \end{array}
                           \right).
\end{align}
Similar to the proof of \eqref{eq3.2}, we have
\begin{align}\label{eq3.3}
JH+\overline{\lambda}J=&\left(
                          \begin{array}{cc}
                            0 & I \\
                            -I & B(A^*-\overline{\lambda})^{-1} \\
                          \end{array}
                        \right)\left(
                                 \begin{array}{cc}
                                   S_1(-\overline{\lambda}) & 0 \\
                                   0 & -A^*+\overline{\lambda} \\
                                 \end{array}
                               \right)\\
\nonumber         &\left(
                     \begin{array}{cc}
                       I & 0 \\
                       (-C(A-\lambda)^{-1})^* & I \\
                     \end{array}
                   \right),
\end{align}
where $S_1(-\overline{\lambda}):=A+\overline{\lambda}+B(A^*-\overline{\lambda})^{-1}C$ is the first Schur complement of $H$.
In the factorization \eqref{eq3.3}, the first and last factor are bounded and boundedly invertible,
and therefore by Lemma \ref{lemA2} and Lemma \ref{lemA3},
\begin{align}\label{eq3.4}
(JH)^*-\lambda J=&\left(
    \begin{array}{cc}
      I & -C(A-\lambda)^{-1} \\
      0 & I \\
    \end{array}
  \right)\left(
           \begin{array}{cc}
             (S_1(-\overline{\lambda}))^* & 0 \\
             0 & -A+\lambda \\
           \end{array}
         \right)\\
\nonumber  &\left(
                    \begin{array}{cc}
                      0 & -I \\
                      I & (B(A^*-\overline{\lambda})^{-1})^* \\
                    \end{array}
                  \right).
\end{align}
Note that in the factorization \eqref{eq3.2} or \eqref{eq3.4}, the first and last factor are bounded and boundedly invertible,
so that $JH=(JH)^*$ if and only if $-S_2(\lambda)=S_1(-\overline{\lambda})^*$.
Similarly, $JH=(JH)^*$ if and only if $S_1(-\overline{\lambda})=(-S_2(\lambda))^*$.
\end{proof}

\begin{corollary}\label{cor3.1}
Suppose that $\mathcal{D}(H)=\mathcal{D}(A)\times\mathcal{D}(A^*)$, and that $\rho(A)\neq\emptyset$.
Then $JH=(JH)^*$ if one of the following holds:
\begin{enumerate}
\item $C$ is $A$-bounded with relative bound $0$,
\item $B$ is $A^*$-bounded with relative bound $0$.
\end{enumerate}
\end{corollary}
\begin{proof}
We prove the claim in case (1); the proof in case (2) is analogous.
Taking $\lambda\in\rho(A)$. It is enough to prove
$$(A+\overline{\lambda}+B(A^*-\overline{\lambda})^{-1}C)^*=A^*+\lambda+C(A-\lambda)^{-1}B.$$
First we note that $B(A^*-\overline{\lambda})^{-1}C$ is $(A+\overline{\lambda})$-bounded with relative bound $0$ since $B(A^*-\overline{\lambda})^{-1}$ is bounded (note that it is a closed everywhere defined operator) and $C$ is $A$-bounded with relative bound $0$.
Next we prove $(B(A^*-\overline{\lambda})^{-1}C)^*$ is $(A+\overline{\lambda})^*$-bounded with relative bound $0$.
To do this, we claim that $C(A-\lambda)^{-1}B$ is $A^*$-bounded with relative bound $0$.
In fact, by the assumptions,
for every $\varepsilon>0$ there is a real number $b(\varepsilon)$ such that for all $x\in\mathcal{D}(A)$,
$$\|Cx\|\leq \varepsilon\|(A-\lambda)x\|+b(\varepsilon)\|x\|,$$
and for some $a$ and $b$ in $\mathbb R$ and all $x\in\mathcal{D}(A^*)$,
$$\|Bx\|\leq a\|A^*x\|+b\|x\|,$$
so for all $x\in\mathcal{D}(A^*)$,
\begin{align*}
\|C(A-\lambda)^{-1}Bx\|&\leq \varepsilon\|Bx\|+b(\varepsilon)\|(A-\lambda)^{-1}Bx\|  \\
&\leq\varepsilon a\|A^*x\|+b(\varepsilon,\lambda)\|x\|
\end{align*}
since $(A-\lambda)^{-1}B$ is bounded on $\mathcal{D}(B)$.
It follows that $C(A-\lambda)^{-1}B$ is $A^*$-bounded with relative bound $0$.
Noting that
\begin{align*}
(B(A^*-\overline{\lambda})^{-1}C)^*=C^*(B(A^*-\overline{\lambda})^{-1})^*=C\overline{(A-\lambda)^{-1}B}
\end{align*}
and that $\mathcal{D}(B)\supset\mathcal{D}(A^*)$,
so $(B(A^*-\overline{\lambda})^{-1}C)^*$ is $(A+\overline{\lambda})^*$-bounded with relative bound $0$.
Hence by Lemma \ref{lemA1},
\begin{align*}
(A+\overline{\lambda}+B(A^*-\overline{\lambda})^{-1}C)^*=&(A+\overline{\lambda})^*+(B(A^*-\overline{\lambda})^{-1}C)^*\\
                                                        =&A^*+\lambda+C\overline{(A-\lambda)^{-1}B}\\
                                                        =&A^*+\lambda+C(A-\lambda)^{-1}B.
\end{align*}

\end{proof}

For the definition and properties of a maximal accretive operator in the following corollary, see \cite[Section IV.4]{Nagy}.
\begin{corollary}\label{cor3.2}
Suppose that $\mathcal{D}(H)=\mathcal{D}(A)\times\mathcal{D}(A^*)$, and that $A$ or $-A$ is maximal accretive.
Then $JH=(JH)^*$ if one of the following holds:
\begin{enumerate}
\item $C$ is $A$-bounded with relative bound $< 1$ and $B$ is $A^*$-bounded with relative bound $\leq 1$,
\item $C$ is $A$-bounded with relative bound $\leq 1$ and $B$ is $A^*$-bounded with relative bound $< 1$.
\end{enumerate}
\end{corollary}
\begin{proof}
Since $-H$ is also a Hamiltonian operator, it is enough to prove the case that $A$ is maximal accretive.
We prove the claim in case (2); the proof in case (1) is analogous.
Note that the operators $A$ and $A^*$ can be maximal accretive only simultaneously.
To prove $JH=(JH)^*$, we have to show that for some $\lambda>0$,
\begin{align}\label{eq3.5}
(A-\lambda+B(A^*+\lambda)^{-1}C)^*=A^*-\lambda+C(A+\lambda)^{-1}B.
\end{align}

\noindent \emph{Step 1}. We start from
the claim that for $\lambda>0$ large enough, $B(A^*+\lambda)^{-1}C$ is $(A-\lambda)$-bounded with relative bound $<1$.
Since $C$ is $A$-bounded with relative bound $\leq 1$, it is enough to prove
\begin{align}\label{eq3.6}
\|B(A^*+\lambda)^{-1}\|<1
\end{align}
for $\lambda>0$ large enough.
We observe that for $x\in\mathcal{D}(A^*)$,
\begin{align}\label{eq3.7}
\|(A^*+\lambda)x\|^2&=\|A^*x\|^2+2\lambda{\rm Re}(A^*x,x)+\lambda^2\|x\|^2\\
\nonumber           &\geq\|A^*x\|^2+\lambda^2\|x\|^2,
\end{align}
where the last inequality follows from the fact that $A^*$ is maximal accretive.
By the assumption that $B$ is $A^*$-bounded with relative bound $< 1$,
there are real numbers $a<1$ and $b$ such that for $x\in\mathcal{D}(A^*)$,
\begin{align}\label{eq3.8}
\|Bx\|\leq a\|A^*x\|+b\|x\|,
\end{align}
so that for $x\in\mathcal{D}(A^*)$, we have, using \eqref{eq3.7} twice and then \eqref{eq3.8},
\begin{align*}
\|Bx\|\leq (a+\frac{b}{\lambda})\|(A^*+\lambda)x\|.
\end{align*}
It is enough to choose $\lambda>0$ large enough such that $a+\frac{b}{\lambda}<1$.

\noindent \emph{Step 2}. In this step, we show that for $\lambda>0$ large enough, $(B(A^*+\lambda)^{-1}C)^*$ is $(A-\lambda)^*$-bounded with relative bound $<1$.
Noting that
\begin{align*}
(B(A^*+\lambda)^{-1}C)^*=C^*(B(A^*+\lambda)^{-1})^*=C\overline{(A+\lambda)^{-1}B}
\end{align*}
and that $\mathcal{D}(B)\supset\mathcal{D}(A^*)$, it is enough to prove for $\lambda>0$ large enough, $C(A+\lambda)^{-1}B$ is $(A^*-\lambda)$-bounded with relative bound $<1$.
Since $C$ is $A$-bounded with relative bound $\leq 1$ and the operator $A$ is maximal accretive,
we have, with arguments similar to the ones used in the proof of \eqref{eq3.6},
for every $\varepsilon>0$ there is a $N(\varepsilon)>0$
such that for $\lambda>N(\varepsilon)$,
\begin{align}\label{eq3.9}
\|C(A+\lambda)^{-1}\|< 1+\varepsilon.
\end{align}
Further, by \eqref{eq3.8}, for $x\in\mathcal{D}(A^*)$,
\begin{align}\label{eq3.10}
\|Bx\|\leq a\|(A^*-\lambda)x\|+(b+a\lambda)\|x\|.
\end{align}
It follows from \eqref{eq3.9} and \eqref{eq3.10} that for $\lambda>N(\varepsilon)$ and $x\in\mathcal{D}(A^*)$,
\begin{align*}
\|C(A+\lambda)^{-1}Bx\|\leq(1+\varepsilon)a\|(A^*-\lambda)x\|+(1+\varepsilon)(b+a\lambda)\|x\|.
\end{align*}
It is enough to choose $\varepsilon>0$ small enough such that $(1+\varepsilon)a<1$.

\noindent \emph{Step 3}. Now \eqref{eq3.5} follows from \emph{Step 1} and \emph{Step 2}
by applying Lemma \ref{lemA1}.
\end{proof}

\begin{corollary}\label{cor3.3}
Suppose that $\mathcal{D}(H)=\mathcal{D}(A)\times\mathcal{D}(A^*)$, and that $A$ is self-adjoint.
Then $JH=(JH)^*$ if one of the following holds:
\begin{enumerate}
\item $C$ is $A$-bounded with relative bound $< 1$ and $B$ is $A^*$-bounded with relative bound $\leq 1$,
\item $C$ is $A$-bounded with relative bound $\leq 1$ and $B$ is $A^*$-bounded with relative bound $< 1$.
\end{enumerate}
\end{corollary}
\begin{proof}
Similar to the proof of Corollary \ref{cor3.2}, we have
\begin{align*}
(A-i\lambda+B(A^*+i\lambda)^{-1}C)^*=A^*+i\lambda+C(A-i\lambda)^{-1}B
\end{align*}
for some $\lambda>0$.
\end{proof}

Finally, we consider the case $\mathcal{D}(H)=\mathcal{D}(C)\times\mathcal{D}(B)$.
\begin{theorem}\label{th3.2}
Suppose $\mathcal{D}(H)=\mathcal{D}(C)\times\mathcal{D}(B)$.
Then the following statements are equivalent:
\begin{enumerate}
\item $H$ is symplectic self-adjoint,
\item $C+\lambda+A^*(B-\lambda)^{-1}A=(C+\overline{\lambda}+A^*(B-\overline{\lambda})^{-1}A)^*$
for some (and hence for all) $\lambda\in\rho(B)$,
\item $B+\lambda+A(C-\lambda)^{-1}A^*=(B+\overline{\lambda}+A(C-\overline{\lambda})^{-1}A^*)^*$
for some (and hence for all) $\lambda\in\rho(C)$.
\end{enumerate}
\end{theorem}
\begin{proof}
Let $\lambda\in\rho(C)$.
By applying Lemma \ref{lem2.1} to $(JH-\lambda)$ and $(JH-\overline{\lambda})$, respectively,
we have, with arguments similar to the ones used in the proof of Theorem \ref{th3.1},
\begin{align*}
JH-\lambda=&\left(
              \begin{array}{cc}
                I & 0 \\
                -A(C-\lambda)^{-1} & I \\
              \end{array}
            \right)\left(
                     \begin{array}{cc}
                       C-\lambda & 0 \\
                       0 & S_2(\lambda) \\
                     \end{array}
                   \right)\\
            &\left(
               \begin{array}{cc}
                 I & (-A(C-\overline{\lambda})^{-1})^* \\
                 0 & I \\
               \end{array}
             \right),\\
(JH)^*-\lambda=&\left(
                   \begin{array}{cc}
                     I & 0 \\
                     -A(C-\lambda)^{-1} & I \\
                   \end{array}
                 \right)\left(
                          \begin{array}{cc}
                            C-\lambda & 0 \\
                            0 & S_2(\overline{\lambda})^* \\
                          \end{array}
                        \right)\\
                 &\left(
                          \begin{array}{cc}
                            I & (-A(C-\overline{\lambda})^{-1})^* \\
                            0 & I \\
                          \end{array}
                        \right),
\end{align*}
where $S_2(\lambda):=-B-\lambda-A(C-\lambda)^{-1}A^*$ is the second Schur complement of $JH$,
so that $JH=(JH)^*$ if and only $S_2(\lambda)=S_2(\overline{\lambda})^*$.
Similarly, $JH=(JH)^*$ if and only if $S_1(\lambda)=S_1(\overline{\lambda})^*$ for some (and hence for all) $\lambda\in\rho(B)$,
where $S_1(\lambda):=C+\lambda+A^*(B-\lambda)^{-1}A$ is the first Schur complement of $JH$.
\end{proof}

\begin{corollary}\label{cor3.4}
Suppose $\mathcal{D}(H)=\mathcal{D}(C)\times\mathcal{D}(B)$.
Then $JH=(JH)^*$ if one of the following holds:
\begin{enumerate}
\item $A$ is $C$-bounded with relative bound $< 1$ and $A^*$ is $B$-bounded with relative bound $\leq 1$,
\item $A$ is $C$-bounded with relative bound $\leq 1$ and $A^*$ is $B$-bounded with relative bound $< 1$.
\end{enumerate}
\end{corollary}
\begin{proof}
Similar to the proof of Corollary \ref{cor3.2}, we have
\begin{align*}
(B-i\lambda+A(C+i\lambda)^{-1}A^*)^*=B+i\lambda+A(C-i\lambda)^{-1}A^*
\end{align*}
for some $\lambda>0$.
\end{proof}

The following example shows that the relative bounds in the assumptions of Corollary \ref{cor3.2}
(Corollary \ref{cor3.3}, Corollary \ref{cor3.4}, respectively) cannot be improved.
\begin{example}
Let $A$ be a nonnegative unbounded self-adjoint operator on $X$.
Note that $A$ is also maximal accretive.
Consider the Hamiltonian operator
$$H:=\left(
       \begin{array}{cc}
         A & A \\
         -A & -A \\
       \end{array}
     \right).$$
It is not difficult to see that $H$ is not closed since $A$ is unbounded,
so that $JH\neq(JH)^*$.
\end{example}

\section{an application to symplectic elasticity}\label{sec4}

Consider the rectangular thin plate bending problem with two opposite edges simply supported.
The basic governing equation in terms of displacement is
\begin{align}\label{eq4.1}
D(\frac{\partial^2}{\partial x^2}+\frac{\partial^2}{\partial y^2})^2w=q,\mbox{~for $0<x<h$ and $0<y<1$},
\end{align}
the boundary conditions for simply supported edges are
\begin{align}\label{eq4.2}
w=0,\ \frac{\partial^2w}{\partial y^2}=0, \mbox{~for~ $y=0$ or $y=1$},
\end{align}
and the boundary conditions for the other two edges are
\begin{align}\label{eq*}
w,\frac{\partial w}{\partial y}=\mbox{given functions, for $x = 0$ or $x = h$},
\end{align}
see \cite[Section 8.1]{YZL}.
To obtain the analytical solution of the above boundary value problem,
the key step is rewriting \eqref{eq4.1} and \eqref{eq4.2} into an operator equation, see \cite{Zhong}.
Let
\begin{align*}
u_1=&\frac{\partial^2w}{\partial x^2}+\frac{\partial^2w}{\partial y^2},
\ \ \ \ u_2=\frac{\partial^3w}{\partial x^2\partial y}+\frac{\partial^3w}{\partial y^3},\\
u_3=&\frac{\partial^3w}{\partial x^3}+\frac{\partial^3w}{\partial x\partial y^2},
\ u_4=-\frac{\partial^3w}{\partial x^2\partial y}-\frac{\partial^3w}{\partial y^3}.
\end{align*}
Then the boundary value problem \eqref{eq4.1},\eqref{eq4.2} becomes \cite[Example 6.3.1]{Wang}
\begin{align}\label{eq4.3}
&\frac{\partial}{\partial x}\left(
                             \begin{array}{c}
                               u_1 \\
                               u_2 \\
                               u_3 \\
                               u_4 \\
                             \end{array}
                           \right)=
                           \left(
                             \begin{array}{cccc}
                               0 & -\frac{\partial}{\partial y} & 1 & -\frac{\partial}{\partial y} \\
                               \frac{\partial}{\partial y} & 0 & \frac{\partial}{\partial y} & 1 \\
                               0 & 0 & 0 & \frac{\partial}{\partial y} \\
                               0 & 0 & -\frac{\partial}{\partial y} & 0 \\
                             \end{array}
                           \right)
                           \left(
                             \begin{array}{c}
                               u_1 \\
                               u_2 \\
                               u_3 \\
                               u_4 \\
                             \end{array}
                           \right)+
                           \left(
                             \begin{array}{c}
                               0 \\
                               0 \\
                               \frac{q}{D} \\
                               0 \\
                             \end{array}
                           \right),\\
\nonumber &u_1=u_3=0 \mbox{~for $y=0$ or $y=1$}.
\end{align}
Next we write \eqref{eq4.3} as an operator equation in a Hilbert space.
Let
\begin{align*}
\mathcal{D}(T_0):=&\{g\in L^2(0,\ 1)~|~g\in AC[0,\ 1], \ g'\in L^2(0,\ 1),\ g(0)=g(1)=0\}, \\
T_0g:=&g' \mbox{~for $g\in\mathcal{D}(T_0)$}.
\end{align*}
Then $T_0$ is a closed densely defined linear operator on the Hilbert space
$L^2(0,\ 1)$ and, furthermore, its adjoint operator is determined by \cite[Example III.5.31]{Ka1980}
\begin{align*}
\mathcal{D}(T_0^*):=&\{g\in L^2(0,\ 1)~|~g\in AC[0,\ 1], \ g'\in L^2(0,\ 1)\}, \\
T_0^*g:=&-g' \mbox{~for $g\in\mathcal{D}(T_0^*)$}.
\end{align*}
Let
$$A:=\left(
     \begin{array}{cc}
       0 & T_0^* \\
       T_0 & 0 \\
     \end{array}
   \right).$$
Then \eqref{eq4.3} becomes $\dot{u}=Hu+f$, where
\begin{align}\label{eq4.4}
H:=\left(
     \begin{array}{cc}
       A & A+1 \\
       0 & -A \\
     \end{array}
   \right)
\end{align}
is a Hamiltonian operator matrix on the Hilbert space $(L^2(0,\ 1))^4$
and
$$u:=(u_1\  u_2 \ u_3\  u_4)^t,\ f:=(0\ 0\ \frac{q}{D}\ 0)^t,$$
so that the spectral properties of the operator $H$ are essential for us to
get the analytical solution of the boundary value problem \eqref{eq4.1}, \eqref{eq4.2}, and \eqref{eq*}.
In the proof of the following proposition, the \emph{essential spectrum} of a closed operator $T$ is defined as \cite[Definition 2.1.9]{Tre}
$$\sigma_{ess}(T):=\{\lambda\in\mathbb C~|~(T-\lambda) \mbox{~is not Fredholm}\}.$$
\begin{proposition}
For the Hamiltonian operator matrix $H$ given by \eqref{eq4.4}, we have:
\begin{enumerate}
\item $H$ is symplectic self-adjoint,
\item $\sigma(H)=\sigma_p(H)=\{k\pi, k\in\mathbb Z\}$ consists of eigenvalues of finite algebraic multiplicity.
\end{enumerate}
\end{proposition}

\begin{proof}
Since $A$ is self-adjoint, the first statement follows from Corollary \ref{cor3.1} or Corollary \ref{cor3.3}.
We start to prove the second statement.
A little calculation shows
$$H^2=\left(
      \begin{array}{cc}
        A^2 & 0 \\
        0 & A^2 \\
      \end{array}
    \right),
    A^2=\left(
        \begin{array}{cc}
          T_0^*T_0 & 0 \\
          0 & T_0T_0^* \\
        \end{array}
      \right).$$
Thus
\begin{align*}
&\sigma(H^2)=\sigma(A^2)=\sigma(T_0^*T_0)\cup\sigma(T_0T_0^*),\\
&\sigma_{ess}(H^2)=\sigma_{ess}(A^2)=\sigma_{ess}(T_0^*T_0)\cup\sigma_{ess}(T_0T_0^*).
\end{align*}
From \cite[Example V.3.25]{Ka1980} we see that the self-adjoint operators $T_0^*T_0, T_0T_0^*$ are determined by
\begin{align*}
\mathcal{D}(T_0^*T_0):=&\{g\in L^2(0,\ 1)~|~g,g'\in AC[0,\ 1], \ g''\in L^2(0,\ 1),\ g(0)=g(1)=0\}, \\
T_0^*T_0g:=&-g'' \mbox{~for $g\in\mathcal{D}(T_0^*T_0)$},\\
\mathcal{D}(T_0T_0^*):=&\{g\in L^2(0,\ 1)~|~g,g'\in AC[0,\ 1], \ g''\in L^2(0,\ 1),\ g'(0)=g'(1)=0\}, \\
T_0T_0^*g:=&-g'' \mbox{~for $g\in\mathcal{D}(T_0T_0^*)$},
\end{align*}
and so
\begin{align*}
&\sigma(T_0^*T_0)=\{(n\pi)^2, n=1, 2, 3,...\},\\
&\sigma(T_0T_0^*)=\{(n\pi)^2, n=0, 1, 2,...\}, \\
&\sigma_{ess}(T_0^*T_0)=\sigma_{ess}(T_0T_0^*)=\emptyset,
\end{align*}
so that
\begin{align}
&\sigma(H^2)=\sigma(T_0T_0^*)=\{(n\pi)^2, n=0, 1, 2,...\}, \label{eq4.5}\\
&\sigma_{ess}(H^2)=\emptyset.\label{eq4.6}
\end{align}
Moreover, $H^2+1$ is bijective since $H^2$ is a nonnegative self-adjoint operator,
so it follows from
$$H^2+1=(H-i)(H+i)=(H+i)(H-i)$$
that $(H-i)$ is bijective, and hence $i\in\rho(H)$.
Then, by \cite[Theorem VII.9.10]{DS} and \cite[Lemma 2]{BG}, respectively,
\begin{align}
&\sigma(H^2)=\{\lambda^2~|~\lambda\in\sigma(H)\}, \label{eq4.7}\\
&\sigma_{ess}(H^2)=\{\lambda^2~|~\lambda\in\sigma_{ess}(H)\}.\label{eq4.8}
\end{align}
Now \eqref{eq4.6} and \eqref{eq4.8} imply that $\sigma_{ess}(H)=\emptyset$ and,
furthermore,
$\sigma(H)$ consists of countably many isolated eigenvalues of finite algebraic multiplicity since $\rho(H)\neq\emptyset$ (see \cite[Theorem XVII.2.1]{GGK}).
But $H$ is symplectic self-adjoint,
so that $\sigma_p(H)=\sigma_p(H)\cup\sigma_r(H)$ is symmetric with respect to the imaginary axis (see \cite[Theorem 3.6]{AHF}),
and therefore $\sigma(H)=\sigma_p(H)=\{k\pi, k\in\mathbb Z\}$ by \eqref{eq4.5} and \eqref{eq4.7}.
\end{proof}

\begin{remark}
The assertion that $H$ is symplectic self-adjoint with $\sigma_c(H)\cup\sigma_r(H)=\emptyset$ is new.
\end{remark}

\begin{remark}
It has been proved that the root vector system of $H$ forms a Schauder basis with parentheses for $(L^2(0,\ 1))^4$,
so that the analytical solution of \eqref{eq4.3} could be obtained by expansion of root vectors,
see \cite[Example 6.3.1]{Wang} for details.
\end{remark}

\appendix
\section{some lemmas on adjoints}

\begin{lemma}(\cite[Corollary 1]{HK1970})\label{lemA1}
Let $T$ be a closed densely defined linear operator on a Hilbert space.
Suppose $S$ is a $T$-bounded operator such that $S^*$ is $T^*$-bounded, with both relative bounds $<1$.
Then $S+T$ is closed and $(S+T)^*=S^*+T^*$.
\end{lemma}

\begin{lemma}\label{lemA2}(\cite[Problem III.5.26]{Ka1980})
Let $S$ be a bounded everywhere defined operator and $T$ be a densely defined operator on a Hilbert space.
Then $(ST)^*=T^*S^*$.
\end{lemma}

\begin{lemma}(\cite[Theorem 4.3]{AD})\label{lemA3}
Let $S$ and $T$ be densely defined operators on a Hilbert space.
If $T$ is closed and $\mathcal{R}(T)$ is closed and has finite codimension, then $ST$ is a
densely defined operator and $(ST)^*=T^*S^*$.
\end{lemma}

{\textbf{Acknowledgment.}}~{\ This work was supported by
the Natural Science Foundation of China (Grant Nos. 11371185, 11101200, 11361034),
the Specialized Research Fund for the Doctoral Program of Higher Education of China (Grant No. 20111501110001),
the Major Subject of Natural Science Foundation of Inner Mongolia of China (Grant No. 2013ZD01),
and the Natural Science Foundation of Inner Mongolia of China (Grant No. 2012MS0105).
The authors wish to thank Gerald Teschl and Guolin Hou for many valuable comments.}

\end{document}